\PassOptionsToPackage{dvipsnames}{xcolor}
\documentclass[11pt,reqno,oneside]{amsart}
\usepackage{amsmath, amssymb, graphicx, mathrsfs, dsfont, hyperref}
\usepackage{float}
\usepackage{amsthm}
\usepackage[top=1in, bottom=1in, left=1in, right=1in]{geometry}
\usepackage{longtable} 
\usepackage{enumitem}
\usepackage{adjustbox}

\theoremstyle{definition}
\newtheorem{theorem}{Theorem}[section]

\newtheorem{corollary}[theorem]{Corollary}

\newtheorem{conjecture}[theorem]{Conjecture}

\newtheorem*{theorem*}{Theorem}
\newtheorem*{conjecture*}{Conjecture}
\newtheorem*{question*}{Question}
\newtheorem*{lemma*}{Lemma}
\newtheorem*{corollary*}{Corollary}
\newtheorem*{definition*}{Definition}
\newtheorem*{example*}{Example}
\usepackage{mathabx}

\DeclareMathOperator{\lcm}{lcm}

\author{Brennan Benfield and Oliver Lippard}
\address{Brennan Benfield: Department of Mathematics and Statistics, University of North Carolina at Charlotte, 9201 University City Blvd., Charlotte, NC 28223, USA}
\address{Oliver Lippard: Department of Mathematics and Statistics, University of North Carolina at Charlotte, 9201 University City Blvd., Charlotte, NC 28223, USA}

\title{Fixed Points of $K$-Fibonacci Sequences}

\begin{document}

\begin{abstract}
    A $K$-Fibonacci sequence is a binary recurrence sequence where $F_0=0$, $F_1=1$, and $F_n=K\cdot F_{n-1}+F_{n-2}$. These sequences are known to be periodic modulo every positive integer greater than $1$. If the length of one shortest period of a $K$-Fibonacci sequence modulo a positive integer is equal to the modulus, then that positive integer is called a \textit{fixed point}. This paper determines the fixed points of $K$-Fibonacci sequences according to the factorization of $K^2+4$ and concludes that if this process is iterated, then every modulus greater than 3 eventually terminates at a fixed point. 
\end{abstract}
\maketitle

The results of this paper were presented at the 21\textsuperscript{st} International Fibonacci Conference, hosted at Harvey Mudd College, July 8, 2024.

\section{Introduction}
Perhaps the most popular sequence of all, the Fibonacci sequence is a binary recurrence defined by $F_0=0$, $F_1=1$, and $F_n=F_{n-1}+F_{n-2}$ and has been studied so thoroughly that discovering new properties at all is almost as surprising as the property discovered. The focus of this paper is a particular property that was first recognized in 1877 by Lagrange \cite{[Lagrange]}: the terms in the Fibonacci sequence modulo $10$ (i.e. the one's place digits) repeat every $60$ terms. Inquiry regarding the periodicity of the sequence modulo a positive integer has continued and in 2003 it was proven by Everest, Shparlinski, and Ward \cite{[Everest]} that \textit{every} binary recurrence sequence is periodic modulo a positive integer $m>1$. 

Define the \textit{Pisano period} as the length of one (shortest) period of the Fibonacci sequence modulo $m$. The Pisano period is denoted $\pi(m)$ and can be recursively applied: denote $\pi(\pi(m))~=~\pi^2(m)$ and $\pi(\pi^i(m))=\pi^{i+1}(m)$ for $n>2$. The \textit{trajectory} of an integer $m$ is the set of integers given by recursively applying the Pisano period to $m$. How long can the trajectory of a starting integer $m$ be? 

A \textit{fixed point} is a positive integer $m$ such that $\pi(m)=m$. A \textit{$k$-periodic point} is a $k$-tuple of positive integers such that $\pi(m)=\pi^2(m)=\ldots=\pi^k(m)=m$. In 1969, Fulton and Morris \cite{[Fulton_Morris]} determined the fixed points of the Fibonacci sequence and proved that the trajectory of every integer will eventually terminate at a fixed point. In 2020, Trojovsk\'a \cite{[Trojovska]} showed that there are no $k$-periodic points in the Fibonacci sequence. 
\begin{theorem}[\textit{Fixed Point Theorem}, Fulton and Morris \cite{[Fulton_Morris]}]\label{Fib_Fixed_Point}
For and integer $m>1$, $\pi(m)=m$ if and only if $m=24\times5^{k-1}$. 
\end{theorem}

\begin{theorem}(\textit{Iteration Theorem}, Fulton and Morris \cite{[Fulton_Morris]})\label{Fib_Iteration}
For every integer $m>1$, there exists a least integer $j$ such that $\pi^{j+1}(m)=\pi^j(m)$.
\end{theorem}

\begin{example*}
%In the Fibonacci sequence, the trajectory of $2$ is given by:
%\[
%\pi(2)=3 \quad \rightarrow \quad \pi(3)=8 \quad \rightarrow \quad \pi(8)=12 \quad \rightarrow \quad \pi(12)=24.
%\] 
\begin{center}
   \text{ Trajectories of the integers in the Fibonacci sequence for $m\varleq24$.}
\end{center}
\begin{align*}
&\pi(2)=3 \rightarrow \pi(3)=8 \rightarrow \pi(8)=12 \rightarrow \pi(12)=24 \qquad &\pi(13)=28 \rightarrow \pi(28)=48 \rightarrow \pi(48)=24\\
&\pi(3)=8 \rightarrow \pi(8)=12 \rightarrow \pi(12)=24 \qquad &\pi(14)=48 \rightarrow \pi(48)=24\\
&\pi(4)=6 \rightarrow \pi(6)=24 \qquad &\pi(15)=40 \rightarrow \pi(40)=60 \rightarrow \pi(60)=120\\
&\pi(5)=20 \rightarrow \pi(20)=60 \rightarrow \pi(60)=120 \qquad &\pi(16)=24\\
&\pi(6)=24 \qquad &\pi(17)=36 \rightarrow \pi(36)=24\\
&\pi(7)=16 \rightarrow \pi(16)=24 \qquad &\pi(18)=24\\
&\pi(8)=12 \rightarrow \pi(12)=24 \qquad &\pi(19)=18 \rightarrow \pi(18)=24\\
&\pi(9)=24 \qquad &\pi(20)=60 \rightarrow \pi(60)=120\\
&\pi(10)=60 \rightarrow \pi(60)=120 \qquad &\pi(21)=16 \rightarrow \pi(16)=24\\
&\pi(11)=10 \rightarrow \pi(10)=60 \rightarrow \pi(60)=120 \qquad &\pi(22)=30 \rightarrow \pi(30)=120\\
&\pi(12)=24 \qquad &\pi(23)=48 \rightarrow \pi(48)=24\\
&&\pi(24)=24
\end{align*}
\end{example*}

Let $\mathcal{T}(m)$ denote the number of iterations of the Pisano period of $m$ until the trajectory reaches a fixed point. For example, $\mathcal{T}(2)=4$ and $\mathcal{T}(24)=0$.
\begin{theorem}\label{trajectory}
    $\limsup \frac{\mathcal{T}(m)}{\log m} < \infty$.
\end{theorem}

Closely related are \textit{$K$-Fibonacci sequences} defined by $F_{K,0}=0$, $F_{K,1}=1$, and $F_{K,n}=KF_{n-1}+F_{n-2}$. The Pisano period of these sequences is commonly denoted $\pi_K(m)$. What are the fixed points of $K$-Fibonacci sequences? Some fixed points have already been calculated: 
\begin{theorem}[Falcon and Plaza \cite{[Falcon]}]
    If $K \vargeq 2$ is even, then $\pi_K(2K)=4$. If $K \vargeq 3$ is odd, then $\pi_K(2K)=6$.
\end{theorem}
\noindent It follows that if $K=2$, then $\pi_K(4)=4$ and if $K=3$, then $\pi_K(6)=6$. The main result of this paper extends the results of Fulton \& Morris \cite{[Fulton_Morris]}, establishing that for each $K$, fixed points may be classified according to one of four categories of $K$. Essentially, for every $K$ the trajectory of every $m>1$ terminates at a fixed point (or is a $k$-periodic point).

\begin{theorem}[$K$-Fixed Point Theorem]\label{K_Fix_Point}
Let $p_1^{e_1}p_2^{e_2}\cdots p_t^{e_t}$ be the prime factorization of $K^2+4$ for a positive integer $K\vargeq1$, then for every modulus $m>1$ and for $j_i=0,1,2,\ldots$, $\pi_K(m)=m$ if and only if:
\begin{enumerate}[label=(roman*)]
    \item $K\equiv\pm1\pmod{6}$ and $m=24\times p_1^{j_1}p_2^{j_2}\cdots p_t^{j_t}$.\\
    \item $K\equiv2\pmod{4}$ and $m=2^{j_1}$ or $m=2^{j_1+1}p_2^{j_2}\cdots p_t^{j_t}$ \quad where \quad $p_1=2$.
    \\
    \item $K\equiv3\pmod{6}$ and $m=6$ or $m=12\times p_3^{j_3}p_4^{j_4}\cdots p_t^{j_t}$.\\
    \item $K\equiv 0 \pmod{4}$ and $m=2$ or $m=4\times p_2^{j_2+1}p_3^{j_3+1}\cdots p_t^{j_t+1}$ \quad where \quad $p_1^{e_1}=2^2$.
    \\
\end{enumerate}
with the only exception when $K\equiv3\pmod{6}$, which has the $2$-periodic points $\pi_K(2)=3$ and $\pi_K(3)=2$.
\end{theorem}

\begin{theorem}[$K$-Iteration Theorem]\label{K-Iteration}
    For every $m>1$, there exists a positive integer $N$ such that $\pi_K^N(m)=m$, with the only exception when $K\equiv3\pmod{6}$, which has the $2$-periodic points $\pi(2)=3$ and $\pi(3)=2$.
\end{theorem}

Note that every positive integer $K\vargeq1$ belongs to exactly one of the four categories in Theorem \ref{K_Fix_Point}. Perhaps it is surprising that the prime factors of $K^2+4$ are important in finding fixed points. Famously, the limit of the ratio between successive Fibonacci numbers is the number $\phi$ where Binet's formula gives the $n$\textsuperscript{th} Fibonacci number: 
\[F_n=\frac{\phi^n-\phi^{-n}}{\sqrt{5}} \qquad \text{where}\ \qquad \phi=\frac{1+\sqrt5}{2}.\] 
Something analogous occurs for $K$-Fibonacci sequences (see \cite{[Falcon],Renault}). Note the similarity when $K=1$:
\[F_{K,n}=\frac{\sigma^n-\sigma^{-n}}{\sqrt{K^2+4}} \qquad \text{where}\ \qquad \sigma=\frac{K+\sqrt{K^2+4}}{2}.\]

\begin{table}[h!]
    \centering
    \begin{adjustbox}{width=1\textwidth}
    %\caption{Fixed Points of $K$-Fibonacci sequences for $K\varleq100$ and $j_i=0,1,2,\ldots$}
    \begin{tabular}{c|c|c|c|c|c}\label{Table 1}
        $K$&$K^2+4$&$K\equiv\pm1\pmod{6}$&$K\equiv2\pmod{4}$&$K\equiv3\pmod{6}$&$K\equiv0\pmod{4}$ \\
         \hline
         $K=1$&5&$24\times5^{j_1}$&&  \\
         \hline
         $K=2$&$2^3$&&$2^{j_1}$&\\
          \hline 
          $K=3$&$13$&&&$6$\ or\ $12\times13^{j_1}$\\
          \hline
          $K=4$&$2^2\cdot5$&&&&$2$\ or\ $4\times5^{j_1+1}$\\
          \hline
          $K=5$&29&$24\times29^{j_1}$&&&\\
          \hline
          $K=6$&$2^3\cdot5$&&$2^{j_1}$\ or\ $2^{j_1+1}\times5^{j_2}$&&\\
          \hline
          $K=7$&$53$&$24\times53^{j_1}$&&&\\
          \hline
          $K=8$&$2^2\cdot17$&&&&$2$\ or\ $4\times17^{j_1+1}$\\
          \hline
          $K=9$&$5\cdot17$&&&$6$\ or\ $12\times5^{j_1}\times17^{j_2}$&\\
          \hline
          $K=10$&$2^3\cdot13$&&$2^{j_1}$\ or\ $2^{j_1+1}\times13^{j_2}$&&\\
          \hline
          $K=11$&$5^3$&$24\times5^{j_1}$&&&\\
          \hline
          $K=12$&$2^2\cdot37$&&&&$2$\ or\ $4\times37^{j_1+1}$\\
          \hline
          $K=13$&$173$&$24\times173^{j_1}$&&&\\
          \hline
          $K=14$&$2^3\cdot5^3$&&$2^{j_1}$\ or\ $2^{j_1+1}\times5^{j_2}$&&\\
          \hline
           $K=15$&$229$&&&$6$\ or\ $12\times229^{j_1}$&\\
          \hline
          $K=16$&$2^2\cdot5\cdot13$&&&&$2$\ or\ $4\times5^{j_1+1}\times13^{j_2+1}$\\
          \hline
          $K=17$&$293$&$24\times293^{j_1}$&&&\\
          \hline
          $K=18$&$2^3\cdot41$&&$2^{j_1}$\ or\ $2^{j_1+1}\times41^{j_2}$&&\\
          \hline
          $K=19$&$5\cdot73$&$24\times5^{j_1}\times73^{j_2}$&&&\\
          \hline
          $K=20$&$2^2\cdot101$&&&&$2$\ or\ $4\times101^{j_1+1}$\\
          \hline
          $K=21$&$5\cdot89$&&&$6$\ or\ $12\times5^{j_1}\times89^{j_2}$&\\
          \hline
          $K=22$&$2^3\cdot61$&&$2^{j_1}$\ or\ $2^{j_1+1}\times61^{j_2}$&&\\
          \hline
          $K=23$&$13\cdot41$&$24\times13^{j_1}\times41^{j_2}$&&&\\
          \hline
          $K=24$&$2^2\cdot5\cdot29$&&&&$2$\ or\ $4\times5^{j_1+1}\times29^{j_2+1}$
    \end{tabular}
    \end{adjustbox}
    \caption{Fixed points of $(K,1,0,1)$-sequences for $k=1,2,\ldots,24$ and $j_i=0,1,2,\ldots$}
   \end{table}

\section{Preliminaries}
Used in the proof of Theorems \ref{K_Fix_Point} and \ref{K-Iteration} are select results for $K$-Fibonacci sequences. Renault \cite{Renault} generalizes a result of Wall \cite{[Wall]} and establishes a particular type of multiplicativity for $\pi_K(m)$ among the prime factors of $m$.
\begin{theorem}[Renault \cite{Renault}]\label{lcm}
       Let $m=p_1^{e_1}p_2^{e_2}\cdots p_r^{e_r}$ be the prime factorization of $m$, then \[\pi_K(m)=\lcm\left[\pi_K(p_1^{e_1}),\pi_K(p_2^{e_2}),\ldots, \pi_K(p_r^{e_r})\right].\]
\end{theorem}
\noindent Another result generalized by Renault \cite{Renault} concerns \textit{Wall-Sun-Sun primes} - conjectured primes $p$ where $\pi(p)=\pi(p^2)$. For the Fibonacci sequence, these primes were originally investigated by Wall \cite{[Wall]} in 1960 but became exceedingly popular in the 1990s when Sun \& Sun \cite{[Sun-Sun]} demonstrated that the first case of Fermat's Last Theorem is false for an exponent $p$ only if $p$ is a Wall-Sun-Sun prime. No Wall-Sun-Sun prime has been found, though infinitely many are conjectured to exist. For $K$-Fibonacci sequences, any prime $p$ such that $\pi_K(p)=\pi_K(p^2)$ is a \textit{$K$-Wall-Sun-Sun prime}. These are much more common: if $K=2$, then $\pi_2(13)=\pi_2(169)$ and $\pi_2(31)=\pi_2(961)$.

\begin{theorem}[Renault \cite{Renault}]\label{W-S-S}
       For all $K$ and all primes $p$, there exists a maximal $e$ such that $\pi_K(p^e) = \pi_K(p)$. For all $x \vargeq e \vargeq 1$, $\pi_K(p^x)=p^{x-e}\cdot\pi_K(p)$.
\end{theorem}  
\begin{theorem}[Renault \cite{Renault}]\label{e1<e2}
    If $\pi_K(p^e)\neq\pi_K(p^{e+1})$, then $\pi_K(p^{e+1})\neq\pi_K(p^{e+2})$.
\end{theorem}   
\noindent For $K$-Fibonacci sequences, Bouazzaoui \cite{Bouazzaoui,Bouazzaoui_2} showed that a prime $p$ that \textit{does not} divide $K^2+4$ is a $K$-Wall-Sun-Sun prime if and only if $\mathbb{Q}\left(\sqrt{K^2+4}\right)$ is not $p$-rational. Results from Harrington \& Jones \cite{Harrington,Jones2023,Jones2024} establish necessary and sufficient conditions for precisely which primes that \textit{do} divide $K^2+4$ are $K$-Wall-Sun-Sun primes. In particular, Harrington \& Jones establish that no odd prime divisor of $K^2+4$ is ever a $K$-Wall-Sun-Sun prime.
\begin{theorem}[Harrington \& Jones \cite{Harrington}]\label{H&J}
    Let $p$ be a prime divisor of $K^2+4$. Then $p$ is a $K$-Wall-Sun-Sun prime if and only if $p=2$ and $K\equiv0\pmod{4}$.
\end{theorem}
\noindent In \cite{BL} the authors proved explicit formulas for $\pi_K(2^e)$, which depend upon the parity of $K$:
\begin{theorem}[\cite{BL}, Theorems 4.25 and 4.28]\label{powersof2}
    Let $K$ be an integer and let $e \vargeq 1$, then:
    \begin{itemize}
        \item If $K$ is odd, then $\pi_K(2^e)=3\cdot2^{e-1}$.
        \item If $K$ is even, then $\pi_K(2^e)=2^{e+1-\nu_2(\gcd(K,2^e))}$ where $\nu_2(x)$ is the 2-adic valuation.
    \end{itemize}
\end{theorem}
\noindent In an effort to characterize the behavior of $\pi_K(p)$ for a prime $p$, Renault \cite{Renault} develops a trichotomy depending on the quadratic residues of $K^2+4$ modulo $p$. 
\begin{theorem}[Renault \cite{Renault}]\label{quad_res}
    Let $p$ be an odd prime. 
    \begin{itemize}
        \item If $K^2+4$ is a nonzero quadratic residue modulo $p$, then $\pi_K(p) \mid  (p-1)$.
        \item If $K^2+4$ is a quadratic non-residue modulo $p$, then $\pi_K(p) \mid  2(p+1)$.
        \item If $p\mid (K^2+4)$, then $\pi_K(p) = p\cdot\text{ord}_p(2^{-1}K)$.
    \end{itemize}   
\end{theorem}
In light of Theorem \ref{quad_res}, it is possible to completely determine $\pi_K(p)$ for primes that divide $K^2+4$.
\begin{theorem}\label{ord=4}
    If $p$ is an odd prime and $p\mid (K^2+4)$, then, $\text{ord}_p(2^{-1}K)=4$ and $\pi_K(p)=4p$.
\end{theorem}
\begin{proof}
    Notice that $K^4 = (K^2+4)(K^2-4) + 16$, hence $K^4 \equiv 16 \pmod{K^2+4}$. It follows that $(2^{-1}K)^4 = \frac{K^4}{16} \equiv 1 \pmod{K^2+4}$. Similarly, if a prime $p \mid (K^2+4)$, then $(2^{-1}K)^4 \equiv 1 \pmod{p}$. Thus, $\text{ord}_p(2^{-1}K)\mid 4$. Notice that $K^2 \equiv -4 \pmod{K^2+4}$. It follows that $(2^{-1}K)^2 = \frac{K^2}{4} \equiv -1 \pmod{K^2+4}$. Hence $(2^{-1}K)^2 \equiv -1 \pmod p$ for all $p\mid (K^2+4)$. The condition that $(2^{-1}K)^2 \equiv -1 \equiv 1 \pmod p$ is equivalent to the condition that $p=2$, contradicting the assumption that $p$ is an odd prime. Thus, if $p\neq2$, then $\text{ord}_p(2^{-1}K)=4$.
\end{proof}
\noindent Harrington \& Jones \cite{Harrington} discovered the value of $\pi_K(2)$ when $K$ is an even integer.
\begin{theorem}[Renault \cite{Renault} and Harrington \& Jones \cite{Harrington}]\label{Keven2}
    If $K$ is even, (i.e. if $2\mid (K^2+4)$), then $\pi_K(2)=2$.
\end{theorem}
\noindent For the proof of the main result, it is important to also understand the nature of $\pi_K(2)$ when $K$ is odd.
\begin{theorem}[Renault \cite{Renault}]\label{Kodd2}
     If $K$ is odd, then $\pi_K(2)=3$.
\end{theorem}

\begin{theorem}\label{K3}
    If $3\mid K$, then $\pi_K(3)=2$. Otherwise, $\pi_K(3)=8$.
\end{theorem}
\begin{proof}
    Consider the $K$-Fibonacci sequence modulo $3$. If $3\mid K$,  the sequence becomes $0, 1,\circlearrowleft$, with a period of $2$. If $K\equiv1\pmod3$, the sequence is $0, 1, 1, 2, 0, 2, 2, 1, 0, 1, \circlearrowleft$
    with a period of $8$. If $K\equiv2\pmod3$, the sequence is $0, 1, 2, 2, 0, 2, 1, 1, 0, 1, \circlearrowleft$
    again with a period of $8$.
\end{proof}
\begin{theorem}\label{corollary3.1}
    For all odd primes $p$, if $q$ is a prime such that $q\mid \pi_K(p)$, then $q \varleq p$, with equality only if $p\mid (K^2+4)$.
\end{theorem}
\begin{proof}
    If $K^2+4$ is a nonzero quadratic residue modulo $p$, then $q\mid (p-1)$ and $q<p$. If $K^2+4$ is a quadratic non-residue, then $q\mid 2(p+1)$. Since $p$ is an odd prime, $p+1$ is composite, and is therefore composed of prime factors smaller than $p$. Finally, if $p\mid(K^2+4)$, then by Theorem \ref{quad_res} $\pi_K(p) = p\cdot\text{ord}_p(2^{-1}K)$. It is clear that $p\mid \pi_K(p)$. Since the integers modulo $p$ form a cyclic group of order at most $p-1$, it follows that $\text{ord}_p(2^{-1}K)\mid (p-1)$. Hence, $q \varleq p$.
\end{proof}

\section{Proof of the K-Fixed Point Theorem}

\begin{theorem}
    If $K\equiv\pm1\pmod{6}$ and $m=24\cdot p_1^{j_1}\cdots p_r^{j_r}$ for $j_i=0,1,2,\ldots$, where $p_1^{a_1}\cdots p_r^{a_r}$ is the prime factorization of $K^2+4$, then $\pi_K(m)=m$.
\end{theorem}
\begin{proof}
    Suppose $K\equiv\pm1\pmod{6}$ and $m=24\cdot p_1^{j_1}\cdots p_r^{j_r}$ for $j_i=0,1,2,\ldots$, where $p_1^{a_1}\cdots p_r^{a_r}$ is the prime factorization of $K^2+4$. By Theorem \ref{lcm}, 
    \begin{align}\label{eq_k=1mod6}
    \pi_K(m)=\lcm\left[\pi_K(2^3), \pi_K(3), \pi_K(p_1^{j_1}), \ldots, \pi_K(p_r^{j_r})\right]
    \end{align}
    By Theorem \ref{H&J}, neither $2$, $3$, nor any $p_i$ is a $K$-Wall-Sun-Sun prime. By Theorems \ref{W-S-S} and \ref{Kodd2}, $\pi_K(2^3)=2^{3-1}\pi_K(2)=2^2\cdot3=12$. By Theorem \ref{K3}, $\pi_K(3)=8$. By Theorem \ref{W-S-S}, $\pi_K(p_i^{j_i})=p_i^{j_i-1}\pi_K(p_i)$ for $i=1, 2, \ldots r$. By Theorem \ref{quad_res}, $\pi_K(p_i)=p_i\cdot\text{ord}_{p_i}(2^{-1}K)$ and by Theorem \ref{ord=4}, $\text{ord}_{p_i}(2^{-1}K)=4$ since $p_i$ is odd and divides $K^2+4$. It follows that 
    \[\pi_K(p_i^{j_i})=p_i^{j_i-1}\pi_K(p_i)=p_i^{j_i-1}\cdot p_i\cdot\text{ord}_{p_i}(2^{-1}K)=4p_i^{j_i}\]
    These may be substituted into equation \eqref{eq_k=1mod6} to obtain:
    \[
    \pi_K(m)=\lcm\left[12,8,4p_1,\ldots,4p_r\right]=24\cdot p_1^{j_1}\cdots p_r^{j_r}=m.
    \]
\end{proof}

\begin{theorem}
    If $K\equiv\pm1\pmod{6}$ and $\pi_K(m)=m$, then $m=24\cdot p_1^{j_1}\cdots p_r^{j_r}$ for $j_i=0,1,2,\ldots$, where $p_1^{a_1}\cdots p_r^{a_r}$ is the prime factorization of $K^2+4$.
\end{theorem}
\begin{proof}
Suppose $K\equiv\pm1\pmod{6}$ and $\pi_K(m)=m$. Let $m=2^{j_1}\cdot 3^{j_2}\cdot p_1^{j_1}\cdots p_r^{j_r}\cdot q_1^{\theta_1}\cdots q_s^{\theta_s}$ where $p_1^{a_1}\cdots p_r^{a_r}$ is the prime factorization of $K^2+4$. Note that if $\theta_\lambda\neq0$ then $q_\lambda \mid m$. By Theorem \ref{lcm}, 
    \begin{align}\label{K=1mod6_P1}
    m=\pi_K(m)=\lcm\left[\pi_K(2^{j_1}), \pi_K(3^{j_2}), \pi_K(p_1^{j_1}),\ldots, \pi_K(p_r^{j_r}), \pi_K(q_1^{\theta_1}),\ldots, \pi_K(q_s^{\theta_s})\right]
    \end{align}
    By Theorem \ref{corollary3.1}, if $\rho \mid \pi_K(q_\lambda^{\theta_\lambda})$, then $\rho<q_\lambda$ for all $1 \varleq i \varleq s$. By Theorem \ref{W-S-S}, $\pi_K(q_\lambda^{\theta_\lambda})=q_\lambda^{\theta_\lambda-1}\pi_K(q_\lambda)$ and by Theorem \ref{quad_res}, $\pi_K(q_\lambda) \mid (q_\lambda-1)$ if $K^2+4$ is a quadratic residue, or $\pi_K(q_\lambda) \mid 2(q_\lambda+1)$ if $K^2+4$ is not a quadratic residue modulo $q_\lambda$. Either way, it follows that $q_\lambda \not \mid \pi_K(q_\lambda)$ and $q_\lambda^{\theta_\lambda-1}$ is the largest factor of $q_\lambda$ that divides $\pi_K(q_\lambda^{\theta_\lambda})$, that is, $q_\lambda^{\theta_\lambda-1}\mid\mid\pi_K(q_\lambda^{\theta_\lambda})$. 

    By Theorem \ref{H&J}, neither $2$, $3$, $p_i$, nor $q_\lambda$ are $K$-Wall-Sun-Sun primes. By Theorems \ref{W-S-S} and \ref{Kodd2}, $\pi_K(2^{j_1})=2^{j_1-1}\pi_K(2)=2^{j_1-1}\cdot3$. By Theorems \ref{W-S-S} and \ref{K3}, $\pi_K(3^{j_2})=3^{j_2-1}\pi_K(3)=3^{j_2-1}\cdot8$. By Theorem \ref{W-S-S}, $\pi_K(p_i^{j_i})=p_i^{j_i-1}\pi_K(p_i)$ for $i=1, 2, \ldots r$. By Theorem \ref{quad_res}, $\pi_K(p_i)=p_i\cdot\text{ord}_{p_i}(2^{-1}K)$ and by Theorem \ref{ord=4}, $\text{ord}_{p_i}(2^{-1}K)=4$ since $p_i$ is odd and divides $K^2+4$. It follows that 
    \[\pi_K(p_i^{j_i})=p_i^{j_i-1}\pi_K(p_i)=p_i^{j_i-1}\cdot p_i\cdot\text{ord}_{p_i}(2^{-1}K)=4p_i^{j_i}\]
    
    Note that $q_\lambda\not\mid\pi_K(2^{j_1})=2^{j_1-1}\cdot3$ and $q_\lambda\not\mid\pi_K(3^{j_2})=3^{j_2-1}\cdot2$, and $q_\lambda\not\mid\pi_K(p_i^{j_i})=4p_i^{j_i}$. Thus, $q_\lambda^{\theta_\lambda-1}\mid\mid m$. But this implies that $q_\lambda^{\theta_\lambda}\not\mid m$, which is a contradiction. Equation \eqref{K=3mod6_P1} is reduced to:
    \begin{align}\label{K=3mod6_P2}
        m=\pi_K(m)=\lcm\left[2^{j_1-1}\cdot3, 3^{j_2-1}\cdot8, 4p_1^{j_1},\ldots,4p_r^{j_r}\right]
    \end{align}
    Recall that by hypothesis, $m=2^{j_1}\cdot 3^{j_2}\cdot p_3^{j_3}\cdots p_r^{j_r}\cdot q_1^{\theta_1}\cdots q_s^{\theta_s}$. Suppose $j_1\vargeq4$, then $2^{j_1-1}$ is the largest power of $2$ in \eqref{K=3mod6_P2}, implying $2^{j_1-1}\mid\mid m$. However, this contradicts the hypothesis $2^{j_1}\mid m$. Suppose $j_2\vargeq2$, then $3^{j_2-1}$ is the largest power of $3$ in \eqref{K=1mod6_P1}, implying $3^{j_2-1}\mid\mid m$. However, this contradicts the hypothesis $3^{j_2}\mid m$. Thus, $j_2=1$, and equation \eqref{K=3mod6_P2} is reduced to:
    \[
    m=\pi_k(m)=\lcm\left[2^{j_1-1}\cdot3, 24, 4p_1^{j_1},\ldots,4p_r^{j_r}\right]=24\cdot p_1^{j_1}\cdots p_r^{j_r}\] where $p_1^{a_1}\cdots p_r^{a_r}$ is the prime factorization of $K^2+4$.
\end{proof}

\begin{theorem}
    If $K\equiv2\pmod{4}$ and $m=2^{j_1}$ or $m=2^{j_1+1}p_2^{j_2}\cdots p_r^{j_r}$ for $j_i=0,1,2,\ldots$, where $p_1^{a_1}\cdots p_r^{a_r}$ is the prime factorization of $K^2+4$ and $p_1=2$, then $\pi_K(m)=m$.
\end{theorem}
\begin{proof}
    Suppose $K\equiv2\pmod{4}$ and let $m=2^{j_1}$. By Theorem \ref{H&J}, neither $2$ nor any $p_i$ is a $K$-Wall-Sun-Sun prime. By Theorem \ref{powersof2}, $\pi_K(2^{j_1})=2^{j_1+1-\min(1,j_1)}=2^{j_1}$.\\
    
    \noindent Suppose $m=2^{j_1+1}\cdot p_2^{j_2} \cdots p_r^{j_r}$ for $j_i=0,1,2,\ldots$, where $p_1^{a_1}\cdots p_r^{a_r}$ is the prime factorization of $K^2+4$. By Theorem \ref{lcm}, 
    \begin{align}\label{eq_K=2mod4}
    \pi_K(m)=\lcm\left[\pi_K(2^{j_1+1}),\pi_K(p_2^{j_2}),\ldots,\pi_K(p_r^{j_r})\right]
    \end{align}
    By Theorem \ref{H&J}, neither $2$ nor any $p_i$ is a $K$-Wall-Sun-Sun prime. By Theorem \ref{powersof2}, $\pi_K(2^{j_1+1})=2^{j_1+1}$. By Theorem \ref{quad_res}, $\pi_K(p_i)=p_i\cdot\text{ord}_{p_i}(2^{-1}K)$ and by Theorem \ref{ord=4}, $\text{ord}_{p_i}(2^{-1}K)=4$ since $p_i$ is odd and divides $K^2+4$. It follows that 
    \[\pi_K(p_i^{j_i})=p_i^{j_i-1}\pi_K(p_i)=p_i^{j_i-1}\cdot p_i\cdot\text{ord}_{p_i}(2^{-1}K)=4p_i^{j_i}\]
    These may be substituted into equation \eqref{eq_K=2mod4} to obtain:
    \[
    \pi(m)=\lcm\left[2^{j_1+1},4p_2^{j_2},\ldots,4p_r^{j_r}\right]=2^{j_1+1}\cdot p_2^{j_2} \cdots p_r^{j_r}=m.
    \]
\end{proof}

\begin{theorem}\label{fixed_point_2mod4}
    If $K\equiv2\pmod{4}$ and $\pi_K(m)=m$, then $m=2^{j_1+1}$ or $m=2^{j_1+2}p_2^{j_2}\cdots p_r^{j_r}$ for $j_i=0,1,2,\ldots$, where $p_1^{a_1}\cdots p_r^{a_r}$ is the prime factorization of $K^2+4$ and $p_1=2$.
\end{theorem}
\begin{proof}
    Suppose $K\equiv2\pmod{4}$ and $\pi_K(m)=m$. Let $m=2^{j_1}\cdot p_2^{j_2}\cdots p_r^{j_r}\cdot q_1^{\theta_1}\cdots q_s^{\theta_s}$ where $p_1^{j_1}\cdots p_r^{j_r}$ is the prime factorization of $K^2+4$ and $p_1=2$. Note that if $\theta_\lambda\neq0$ then $q_\lambda \mid m$. By Theorem \ref{lcm}, 
    \begin{align}\label{K=2mod4_P1}
        m=\pi(m)=\lcm\left[\pi_K(2^{j_1}), \pi_K(p_2^{j_2}),\ldots, \pi_K(p_r^{j_r}), \pi_K(q_1^{\theta_1}),\ldots, \pi_K(q_s^{\theta_s})\right]
    \end{align}
    By Theorem \ref{corollary3.1}, if $\rho \mid \pi_K(q_\lambda^{\theta_\lambda})$, then $\rho<q_\lambda$ for all $1 \varleq i \varleq s$. By Theorem \ref{W-S-S}, $\pi_K(q_\lambda^{\theta_\lambda})=q_\lambda^{\theta_\lambda-1}\pi_K(q_\lambda)$ and by \ref{quad_res}, $\pi_K(q_\lambda) \mid (q_\lambda-1)$ if $K^2+4$ is a quadratic residue and $\pi_K(q_\lambda) \mid 2(q_\lambda+1)$ if $K^2+4$ is not a quadratic residue modulo $q_\lambda$. Either way, it follows that $q_\lambda^{\theta_\lambda-1}$ is the largest factor of $q_\lambda$ that divides $\pi_K(q_\lambda^{\theta_\lambda})$, that is, $q_\lambda^{\theta_\lambda-1}\mid\mid\pi_K(q_\lambda^{\theta_\lambda})$. 

    By Theorem \ref{H&J}, neither $2$, $p_i$, nor $q_\lambda$ are $K$-Wall-Sun-Sun primes. By Theorems \ref{W-S-S} and \ref{Keven2}, $\pi_K(2^{j_1})=2^{j_1-1}\pi_K(2)=2^{j_1-1}\cdot2=2^{j_1}$.
    By Theorem \ref{W-S-S}, $\pi_K(p_i^{j_i})=p_i^{j_i-1}\pi_K(p_i)$ for $i=1, 2, \ldots r$. By Theorem \ref{quad_res}, $\pi_K(p_i)=p_i\cdot\text{ord}_{p_i}(2^{-1}K)$ and by Theorem \ref{ord=4}, $\text{ord}_{p_i}(2^{-1}K)=4$ since $p_i$ is odd and divides $K^2+4$. It follows that 
    \[\pi_K(p_i^{j_i})=p_i^{j_i-1}\pi_K(p_i)=p_i^{j_i-1}\cdot p_i\cdot\text{ord}_{p_i}(2^{-1}K)=4p_i^{j_i}\]
    
    Note that $q_\lambda\not\mid\pi_K(2^{j_1})=2^{j_1}$ and $q_\lambda\not\mid\pi_K(3^{j_2})=3^{j_2-1}\cdot8$, and $q_\lambda\not\mid\pi_K(p_i^{j_i})=4p_i^{j_i}$. Thus, $q_\lambda^{\theta_\lambda-1}\mid\mid m$. But this implies that $q_\lambda^{\theta_\lambda}\not\mid m$, which is a contradiction. Equation \eqref{K=2mod4_P1} is reduced to:
    \[
        m=\pi_K(m)=\lcm\left[2^{j_1}, 4p_3^{j_3},\ldots,4p_r^{j_r}\right] = 2^{j_1}\ \text{or}\ 2^{j_1+1}p_2^{j_2}\cdots p_r^{j_r} 
    \]
    where $p_1^{a_1}\cdots p_r^{a_r}$ is the prime factorization of $K^2+4$ and $p_1=2$.
\end{proof}

\begin{theorem}
    If $K\equiv3\pmod{6}$ and $m=6$ or $m=12\cdot p_3^{j_3}\cdots p_r^{j_r}$ for $j_i=0,1,2,\ldots$, where $p_1^{a_1}\cdots p_r^{a_r}$ is the prime factorization of $K^2+4$ and $p_1^{a_1}=2^2$ and $p_2^{a_2}=3$, then $\pi_K(m)=m$.
\end{theorem}
\begin{proof}
    Suppose $K\equiv3\pmod{6}$ and $m=6$. Then by Theorem \ref{lcm}, $\pi_K(m)=\lcm\left[\pi_K(2),\pi_K(3)\right]$. By Theorem \ref{Kodd2}, $\pi_K(2)=3$. By Theorem \ref{K3}, $\pi_K(3)=2$. Hence, $\pi_K(6)=\lcm\left[3,2\right]=6$. Note, this also establishes the $2$-periodic point $\pi_K(2)=3 \rightarrow \pi_K(3)=2\circlearrowleft$.\\

    \noindent Suppose $m=12\cdot p_3^{j_3}\cdots p_r^{j_r}$ for $j_i=0,1,2,\ldots$, where $p_1^{a_1}\cdots p_r^{a_r}$ is the prime factorization of $K^2+4$ and $p_1^{a_1}=2^2$ and $p_2^{a_2}=3$. By Theorem \ref{H&J}, neither $2$, $3$, nor any $p_i$ is a $K$-Wall-Sun-Sun prime. By Theorem \ref{lcm}, 
    \begin{align}\label{eq_K=3mod6}
        \pi_K(m)=\lcm\left[\pi_K(2^2),\pi_K(3),\pi_K(p_3^{j_3}),\ldots,\pi_K(p_r^{j_r)}\right]
    \end{align}
    By Theorem \ref{powersof2}, $\pi_K(2^2)=2\cdot3=6$. By Theorem \ref{K3}, $\pi_K(3)=2$. By Theorem \ref{W-S-S}, $\pi_K(p_i^{j_i})=p_i^{j_i-1}\pi_K(p_i)$ for $i=1, 2, \ldots r$. By Theorem \ref{quad_res}, $\pi_K(p_i)=p_i\cdot\text{ord}_{p_i}(2^{-1}K)$ and by Theorem \ref{ord=4}, $\text{ord}_{p_i}(2^{-1}K)=4$ since $p_i$ is odd and divides $K^2+4$. It follows that 
    \[\pi_K(p_i^{j_i})=p_i^{j_i-1}\pi_K(p_i)=p_i^{j_i-1}\cdot p_i\cdot\text{ord}_{p_i}(2^{-1}K)=4p_i^{j_i}\]
    These may be substituted into equation \eqref{eq_K=3mod6} to obtain:
    \[
    \pi_K(m)=\lcm\left[6,2,4p_3,\ldots,4p_r\right]=12\cdot p_3^{j_3}\cdots p_r^{j_r}=m.
    \]
\end{proof}

\begin{theorem}
    If $K\equiv3\pmod{6}$ and $\pi_K(m)=m$, then $m=6$ or $m=12\cdot p_3^{j_3}\cdots p_r^{j_r}$ for $j_i=0,1,2,\ldots$, where $p_1^{a_1}\cdots p_r^{a_r}$ is the prime factorization of $K^2+4$ and $p_1^{a_1}=2^2$ and $p_2^{a_2}=3$.
\end{theorem}

\begin{proof}
    Suppose $K\equiv3\pmod{6}$ and $\pi_K(m)=m$. Let $m=2^{j_1}\cdot 3^{j_2}\cdot p_3^{j_3}\cdots p_r^{j_r}\cdot q_1^{\theta_1}\cdots q_s^{\theta_s}$ where $2^{2}\cdot3\cdot p_3^{j_3} \cdots p_r^{j_r}$ is the prime factorization of $K^2+4$. Note that if $\theta_\lambda\neq0$ then $q_\lambda \mid m$. By Theorem \ref{lcm}, 
    \begin{align}\label{K=3mod6_P1}
        m=\pi(m)=\lcm\left[\pi_K(2^{j_1})\, \pi_K(3^{j_2})\, \pi_K(p_3^{j_3}),\ldots, \pi_K(p_r^{j_r}), \pi_K(q_1^{\theta_1}),\ldots, \pi_K(q_s^{\theta_s})\right]
    \end{align}
    By Theorem \ref{corollary3.1}, if $\rho \mid \pi_K(q_\lambda^{\theta_\lambda})$, then $\rho<q_\lambda$ for all $1 \varleq i \varleq s$. By Theorem \ref{W-S-S}, $\pi_K(q_\lambda^{\theta_\lambda})=q_\lambda^{\theta_\lambda-1}\pi_K(q_\lambda)$ and by \ref{quad_res}, $\pi_K(q_\lambda) \mid (q_\lambda-1)$ if $K^2+4$ is a quadratic residue and $\pi_K(q_\lambda) \mid 2(q_\lambda+1)$ if $K^2+4$ is not a quadratic residue modulo $q_\lambda$. Either way, it follows that $q_\lambda^{\theta_\lambda-1}$ is the largest factor of $q_\lambda$ that divides $\pi_K(q_\lambda^{\theta_\lambda})$, that is, $q_\lambda^{\theta_\lambda-1}\mid\mid\pi_K(q_\lambda^{\theta_\lambda})$. 

    By Theorem \ref{H&J}, neither $2$, $3$, $p_i$, nor $q_\lambda$ are $K$-Wall-Sun-Sun primes. By Theorem \ref{powersof2}, $\pi_K(2^{j_1})=2^{j_1-1}\cdot3$. By Theorems \ref{W-S-S} and \ref{K3}, $\pi_K(3^{j_2})=3^{j_2-1}\pi_K(3)=3^{j_2-1}\cdot2$. By Theorem \ref{W-S-S}, $\pi_K(p_i^{j_i})=p_i^{j_i-1}\pi_K(p_i)$ for $i=1, 2, \ldots r$. By Theorem \ref{quad_res}, $\pi_K(p_i)=p_i\cdot\text{ord}_{p_i}(2^{-1}K)$ and by Theorem \ref{ord=4}, $\text{ord}_{p_i}(2^{-1}K)=4$ since $p_i$ is odd and divides $K^2+4$. It follows that 
    \[\pi_K(p_i^{j_i})=p_i^{j_i-1}\pi_K(p_i)=p_i^{j_i-1}\cdot p_i\cdot\text{ord}_{p_i}(2^{-1}K)=4p_i^{j_i}\]
    
    Note that $q_\lambda\not\mid\pi_K(2^{j_1})=2^{j_1-1}\cdot3$ and $q_\lambda\not\mid\pi_K(3^{j_2})=3^{j_2-1}\cdot2$, and $q_\lambda\not\mid\pi_K(p_i^{j_i})=4p_i^{j_i}$. Thus, $q_\lambda^{\theta_\lambda-1}\mid\mid m$. But this implies that $q_\lambda^{\theta_\lambda}\not\mid m$, which is a contradiction. Equation \eqref{K=3mod6_P1} is reduced to:
    \begin{align}\label{K=3mod6_P2_2}
        m=\pi_K(m)=\lcm\left[2^{j_1-1}\cdot3, 3^{j_2-1}\cdot2, 4p_3^{j_3},\ldots,4p_r^{j_r}\right]
    \end{align}
    Recall that by hypothesis, $m=2^{j_1}\cdot 3^{j_2}\cdot p_3^{j_3}\cdots p_r^{j_r}\cdot q_1^{\theta_1}\cdots q_s^{\theta_s}$. Suppose $j_1\vargeq3$, then $2^{j_1-1}$ is the largest power of $2$ in \eqref{K=3mod6_P2}, implying $2^{j_1-1}\mid\mid m$. However, this contradicts the hypothesis $2^{j_1}\mid m$. Thus, $j_1=1$ or $2$. Suppose $j_2\vargeq2$, then $3^{j_2-1}$ is the largest power of $3$ in \eqref{K=3mod6_P2}, implying $3^{j_2-1}\mid\mid m$. However, this contradicts the hypothesis $3^{j_2}\mid m$. Thus, $j_2=1$, and equation \eqref{K=3mod6_P2} is reduced to:
    \[
    m=\pi_k(m)=\lcm\left[(6\ \text{or}\ 12), 6, 4p_3^{j_3},\ldots,4p_r^{j_r}\right]= 12\cdot p_3^{j_3}\cdots p_r^{j_r}\ \text{or}\ 6\] where $p_1^{a_1}\cdots p_r^{a_r}$ is the prime factorization of $K^2+4$ and $p_1^{a_1}=2^2$ and $p_2^{a_2}=3$.
\end{proof}

\begin{theorem}
    If $K\equiv0\pmod{4}$ and $m=2$ or $m=4\cdot p_2^{j_2}\cdots p_r^{j_r}$ for $j_i=0,1,2,\ldots$, where $p_1^{a_1}\cdots p_r^{a_r}$ is the prime factorization of $K^2+4$ and $p_1^{a_1}=2^2$, then $\pi_K(m)=m$.
\end{theorem}

\begin{proof}
    Suppose $K\equiv0\pmod{4}$ and $m=2$. By Theorem \ref{Keven2}, $\pi_K(2)=2$. Suppose $m=4\cdot p_2^{j_2}\cdots p_r^{j_r}$ for $j_i=0,1,2,\ldots$, where $p_1^{a_1}\cdots p_r^{a_r}$ is the prime factorization of $K^2+4$ and $p_1^{j_1}=2^2$. By Theorem \ref{lcm}, 
    \begin{align}\label{eq_K=0mod4}
        \pi_K(m)=\lcm\left[\pi_K(2^2),\pi_K(p_2^{j_2}),\ldots,\pi_K(p_r^{j_r})\right].
    \end{align}
    By Theorem \ref{H&J}, $2$ is a $K$-Wall-Sun-Sun prime, so $\pi_K(2^2)=\pi_K(2)=2$. Also by Theorem \ref{H&J}, no other $p_i$ is a $K$-Wall-Sun-Sun prime. By Theorem \ref{W-S-S}, $\pi_K(p_i^{j_i})=p_i^{j_i-1}\pi_K(p_i)$ for $i=1, 2, \ldots r$. By Theorem \ref{quad_res}, $\pi_K(p_i)=p_i\cdot\text{ord}_{p_i}(2^{-1}K)$ and by Theorem \ref{ord=4}, $\text{ord}_{p_i}(2^{-1}K)=4$ since $p_i$ is odd and divides $K^2+4$. It follows that 
    \[\pi_K(p_i^{j_i})=p_i^{j_i-1}\pi_K(p_i)=p_i^{j_i-1}\cdot p_i\cdot\text{ord}_{p_i}(2^{-1}K)=4p_i^{j_i}\]
    These may be substituted into equation \eqref{eq_K=0mod4} to obtain:
    \[
    \pi_K(m)=\lcm\left[2,4p_2,\ldots,4p_r\right]=4\cdot p_2^{j_2}\cdots p_r^{j_r}=m.
    \]
\end{proof}

\begin{theorem}
    If $\pi_K(m)=m$ and $K\equiv0\pmod{4}$, then $m=2$ or $m=4\cdot p_2^{j_2}\cdots p_r^{j_r}$ for $j_i=0,1,2,\ldots$, where $p_1^{a_1}\cdots p_r^{a_r}$ is the prime factorization of $K^2+4$ and $p_1^{a_1}=2^2$.
\end{theorem}

\begin{proof}
Suppose $K\equiv0\pmod{4}$ and $\pi_K(m)=m$. Let $m=2^{j_1}\cdot 3^{j_2}\cdot p_1^{j_1}\cdots p_r^{j_r}\cdot q_1^{\theta_1}\cdots q_s^{\theta_s}$ where $p_1^{j_1}\cdots p_r^{j_r}$ is the prime factorization of $K^2+4$. Note that if $\theta_\lambda\neq0$ then $q_\lambda \mid m$. By Theorem \ref{lcm}, 
    \begin{align}\label{K=0mod4_P1}
        m=\pi(m)=\lcm\left[\pi_K(2^{j_1})\, \pi_K(3^{j_2})\, \pi_K(p_1^{j_1}),\ldots, \pi_K(p_r^{j_r}), \pi_K(q_1^{\theta_1}),\ldots, \pi_K(q_s^{\theta_s})\right]
    \end{align}
    By Theorem \ref{corollary3.1}, if $\rho \mid \pi_K(q_\lambda^{\theta_\lambda})$, then $\rho<q_\lambda$ for all $1 \varleq i \varleq s$. By Theorem \ref{W-S-S}, $\pi_K(q_\lambda^{\theta_\lambda})=q_\lambda^{\theta_\lambda-1}\pi_K(q_\lambda)$ and by Theorem \ref{quad_res}, $\pi_K(q_\lambda) \mid (q_\lambda-1)$ if $K^2+4$ is a quadratic residue, or $\pi_K(q_\lambda) \mid 2(q_\lambda+1)$ if $K^2+4$ is not a quadratic residue modulo $q_\lambda$. Either way, it follows that $q_\lambda^{\theta_\lambda-1}$ is the largest factor of $q_\lambda$ that divides $\pi_K(q_\lambda^{\theta_\lambda})$, that is, $q_\lambda^{\theta_\lambda-1}\mid\mid\pi_K(q_\lambda^{\theta_\lambda})$. 

    By Theorem \ref{H&J}, $2$ is a $K$-Wall-Sun-Sun prime where $\pi_K(2)=\pi_K(2^2)=2$ but neither $3$, $p_i$, nor $q_\lambda$ are $K$-Wall-Sun-Sun primes. By Theorem \ref{powersof2}, for $j_1 \vargeq 2, \pi_K(2^{j_1}) \mid 2^{j_1-1}$. Specifically, if $a=2^{\nu_2(\gcd(K,m))}-1 \vargeq 1$, then $\pi_K(2^{j_1})=2^{j_1-a}$. If $j_1=1, \pi_K(2)=2$, thus forming a fixed point. By Theorems \ref{W-S-S} and \ref{K3}, $\pi_K(3^{j_2})=3^{j_2-1}\pi_K(3)=3^{j_2-1}\cdot2$ if $3\mid K$ and $\pi_K(3^{j_2})=3^{j_2-1}\pi_K(3)=3^{j_2-1}\cdot8$ if $3\not\mid K$. By Theorem \ref{W-S-S}, $\pi_K(p_i^{j_i})=p_i^{j_i-1}\pi_K(p_i)$ for $i=1, 2, \ldots r$. By Theorem \ref{quad_res}, $\pi_K(p_i)=p_i\cdot\text{ord}_{p_i}(2^{-1}K)$ and by Theorem \ref{ord=4}, $\text{ord}_{p_i}(2^{-1}K)=4$ since $p_i$ is odd and divides $K^2+4$. It follows that 
    \[\pi_K(p_i^{j_i})=p_i^{j_i-a}\pi_K(p_i)=p_i^{j_i-1}\cdot p_i\cdot\text{ord}_{p_i}(2^{-1}K)=4p_i^{j_i}\]
    
    Note that $q_\lambda\not\mid\pi_K(2^{j_1-1})=2^{j_1}$ and $q_\lambda\not\mid\pi_K(3^{j_2})=3^{j_2-1}\cdot2\ \text{or}\ 8$, and $q_\lambda\not\mid\pi_K(p_i^{j_i})=4p_i^{j_i}$. Thus, $q_\lambda^{\theta_\lambda-1}\mid\mid m$. But this implies that $q_\lambda^{\theta_\lambda}\not\mid m$, which is a contradiction. Equation \eqref{K=0mod4_P1} is reduced to:
    \begin{align}\label{K=0mod4_P2}
        m=\pi_K(m)=\lcm\left[2^{j_1-a}, 3^{j_2-1}\cdot(2\ \text{or}\ 8), 4p_2^{j_1},\ldots,4p_r^{j_r}\right]
    \end{align}
    Recall that by hypothesis, if $4\mid m$, then $m=2^{j_1-a}\cdot  p_2^{j_2}\cdots p_r^{j_r}\cdot q_1^{\theta_1}\cdots q_s^{\theta_s}$. Suppose $j_1\vargeq3$, then $2^{j_1-1}$ is the largest power of $2$ in \eqref{K=3mod6_P2}, implying $2^{j_1-a}\mid\mid m$. Because $a \vargeq 1$, this contradicts the hypothesis that $2^{j_1}\mid m$. Thus, equation \eqref{K=0mod4_P2} is reduced to:
    \[
    m=\pi_k(m)=\lcm\left[2^{j_1-a} ,4p_2^{j_2},\ldots,4p_r^{j_r}\right]=4\cdot p_2^{j_2}\cdots p_r^{j_r} 
    \]
    where $p_1^{a_1}\cdots p_r^{a_r}$ is the prime factorization of $K^2+4$ and $p_1^{a_1}=2^2$.
\end{proof}

\section{Proof of the K-Iteration Theorem}

\begin{theorem}[Renault \cite{Renault}]\label{Iteration3}
    For all $m > 2$ and $K \vargeq 1$, $\pi_K(m)$ is even.
\end{theorem}

\begin{theorem}\label{2-not-WSS}
    If $K$ is odd, then $2$ is not a $K$-Wall-Sun-Sun prime.
\end{theorem}
\begin{proof}
    Note that by Theorem \ref{Kodd2} that $\pi_K(2)=3$ for odd $K$. It remains to show that $\pi_K(4)\neq3$ for odd $K$. If $K \equiv 1 \pmod 4$, then the $K$-Fibonacci sequence (modulo 4) becomes 0, 1, 1, 2, 3, 1, 0, 1, $\circlearrowleft$, with period 6. If $K \equiv 3 \pmod 4$, then the $K$-Fibonacci sequence (modulo 4) becomes 0, 1, 3, 2, 1, 1, 0, 1, $\circlearrowleft$, again with period 6. Hence, $\pi_K(2)\neq\pi_K(4)$ for all odd $K$.
\end{proof}
\begin{theorem}\label{Iteration1}
   If $p\mid(K^2+4)$, then for all $N$, $p\mid\pi_K^N(p^j)$. 
\end{theorem}

\begin{proof}
    Suppose $p\mid(K^2+4)$. By Theorem \ref{H&J}, $p$ is a $K$-Wall-Sun-Sun prime if and only if $p=2$ and $K\equiv0\pmod{4}$. In this case, by Theorem \ref{Keven2} $\pi_K(2)=\pi_K(4)=2$. Hence, $2\mid\pi_K^N(p)$ for all $N$. If $p=2$ and $K\equiv 2 \pmod 4$, then $\pi_K^N(p^j) = 2^j$ for all $j$, because this is precisely one family of $K$-fixed points by Theorem \ref{fixed_point_2mod4}. On the other hand, suppose that $p$ is an odd prime. By Theorems \ref{W-S-S} and \ref{quad_res}, $\pi_K(p^{j})=p^{j-1}\pi_K(p_i)=4p$. Recursively applying the Pisano period yields $\pi_K^2(p)=\pi_K\left(4p\right)=\lcm\left[\pi_K(2^2),\pi_K(p)\right]=\lcm\left[\pi_K(2^2),4p\right]$. Iterating the Pisano period always preserves a factor of $p$ by induction. Thus, for all $N$, $p\mid\pi_K^N(p)$.
\end{proof}

\begin{theorem}\label{Iteration2}
    If $q\not\mid(K^2+4)$, then there exists a positive integer $N$ such that for all $n\vargeq N$,  $q\not\mid\pi_K^n(q^\theta)$ for any $\theta\vargeq1$.
\end{theorem}
\begin{proof}
    Suppose $q\not\mid(K^2+4)$. By Theorem \ref{W-S-S}, there is a maximal $e$ such that $\pi_K(q^e)=\pi_K(q)$ and for $x\vargeq e \vargeq1$, $\pi_K(q^x)=q^{x-e}\pi_K(q)$. Applying the Pisano period again yields
    
    \begin{align*}
        \pi_K^2(q^{\theta})&=\pi_K\left(q^{x-e}\pi_K(q)\right)=\lcm\left[\pi_K(q^{x-e}),\pi_K^2(q)\right]=\lcm\left[q^{x-e-1}\pi_k(q),\pi_K^2(q)\right].
    \end{align*}
    By Theorem \ref{quad_res}, since $q_i\not\mid(K^2+4)$, it follows that $q^\theta\not\mid\pi_K(q)$ for any positive integer $\theta$. As the Pisano period is recursively applied, the highest power of $q$ that carries to the next iteration is strictly decreasing. Applying the Pisano period precisely $x-e$ times ensures that the highest power of $q$ that emerges is $q^{x-e-(x-e)}$. Hence, there exists a positive integer $N$ such that for all $n\vargeq N$,  $q\not\mid\pi_K^n(q^\theta)$.
\end{proof}

\begin{theorem}\label{iteration-1mod6}
    If $K\equiv\pm1\pmod{6}$ and $m>1$, then there exists an $N$ such that for all $n\vargeq N$, $\pi_K^n(m)=24\cdot p_1^{j_1}\cdots p_r^{j_r}$ for $j_i=0,1,2,\ldots$, where $p_1^{a_1}\cdots p_r^{a_r}$ is the prime factorization of $K^2+4$. 
\end{theorem}
\begin{proof}
    Suppose $K\equiv \pm1 \pmod6$ and $m=2^{s_1}3^{s_2}p_1^{j_1}\cdots p_r^{j_r} q_1^{\lambda_t}\cdots q_r^{\lambda_t}$ for $s_i,j_i\vargeq0$ and $\lambda_i\vargeq1$, where $p_1^{a_1}\cdots p_r^{a_r}$ is the prime factorization of $K^2+4$. By Theorem \ref{lcm}, 
    \begin{align}\label{K=1mod6_P1_2}
    \pi_K(m)=\lcm\left[\pi_K(2^{s_1}), \pi_K(3^{s_2}), \pi_K(p_1^{j_1}),\ldots, \pi_K(p_r^{j_r}), \pi_K(q_1^{\lambda_1}),\ldots, \pi_K(q_s^{\lambda_2})\right]
    \end{align}
    By Theorem \ref{Iteration3}, $2\mid\pi_K(m)$. By Theorem \ref{Kodd2}, $\pi_K(2)=3$ implying that $3\mid\pi_K^2(m)$. By Theorem \ref{K3}, $\pi_K(3)=8$, implying that $2^3\mid\pi_K^3(m)$. By Theorems \ref{W-S-S} and \ref{Keven2}, $\pi_K(2^3)=2^{3-1}\pi_K(2)=4\cdot3=12$, implying $12\mid\pi_K^4(m)$. By Theorems \ref{lcm}, \ref{Keven2}, and \ref{K3}, $\pi_K(12)=\lcm\left[\pi_K(2^2),\pi_K(3)\right]=\lcm\left[6,8\right]=24$, implying $24\mid\pi_K^5(m)$. By Theorem \ref{Iteration1}, if any $j_i\vargeq1$, then the corresponding $p_i$ is preserved for all $N$, and $p_i\mid\pi_K(m)$. By Theorem \ref{Iteration2}, for any $q_i\not\mid(K^2+4)$, $q_i$ vanishes for some sufficiently large $N$, and $q_i\not\mid\pi_K(m)$. Hence, for some $N$, $\pi_K^n(m)=24\cdot p_1^{j_1}\cdots p_r^{j_r}$ where $p_1^{a_1}\cdots p_r^{a_r}$ is the prime factorization of $K^2+4$. 
\end{proof}

\begin{theorem}
    If $K\equiv3\pmod{6}$ and $m>3$, then there exists an $N$ such that for all $n\vargeq N$, $\pi_K^n(m)=p_1^{j_1}\cdots p_r^{j_r}$ for $j_i=0,1,2,\ldots$, where $p_1^{a_1}\cdots p_r^{a_r}$ is the prime factorization of $K^2+4$. And if $m=2$ or $m=3$, then $\pi_K(2)=3$ and $\pi_K(3)=2$.
\end{theorem}

\begin{proof}
    Suppose $K\equiv 3 \pmod6$ and $m=p_1^{j_1}\cdots p_r^{j_r} q_1^{\lambda_t}\cdots q_r^{\lambda_t}>3$ for $j_i\vargeq0$ and $\lambda_i\vargeq1$, where $p_1^{a_1}\cdots p_r^{a_r}$ is the prime factorization of $K^2+4$. By Theorem \ref{lcm}, 
    \begin{align}\label{K=1mod6_P1_3}
    \pi_K(m)=\lcm\left[\pi_K(p_1^{j_1}),\ldots, \pi_K(p_r^{j_r}), \pi_K(q_1^{\lambda_1}),\ldots, \pi_K(q_s^{\lambda_2})\right]
    \end{align}
    By Theorem \ref{Iteration3}, $2\mid\pi_K(m)$. If $\pi_K(m)=2$, then the $K$-Fibonacci sequence must reduce to $0, 1, K\equiv0\pmod m, \circlearrowleft$, so $m \mid K$. In this case, the iterative Pisano period cycles at the 2-periodic point $\pi(2)=3, \pi(3)=2, \circlearrowleft$. Otherwise, either $\pi_K(m)= 2^{\alpha} \cdot \beta$, where $\alpha \vargeq 1$ and $\beta$ is an odd number greater than 1, or $\pi_K(m)=2^\gamma$, where $\gamma \vargeq 2$. By Theorem \ref{Kodd2}, $\pi_K(2)=3$. If $\pi_K(m)=2^\alpha \cdot \beta$, then $\pi_K(m)=\lcm\left[\pi_K(2^\alpha),\pi_K(\beta)\right]$. Since $3\mid\pi_K(2^\alpha)$ and $2\mid\pi_K(\beta)$, $6\mid\pi_K^2(m)$. Else, if $\pi_K(m)=2^\gamma$, and since 2 is not a $K$-Wall-Sun-Sun prime by Theorem \ref{2-not-WSS}, then $\pi_K^2(m)=3\cdot 2^{\gamma-1}$, so $6 \mid \pi_K^2(m)$. By Theorem \ref{Iteration1}, if any $j_i\vargeq1$, then the corresponding $p_i$ is preserved for all $N$, and $p_i\mid\pi_K(m)$. By Theorem \ref{quad_res}, $\pi_K(p_i)=p_i\cdot\text{ord}_{p_i}(2^{-1}K)$ and by Theorem \ref{ord=4}, $\text{ord}_{p_i}(2^{-1}K)=4$ since $p_i$ is odd and divides $K^2+4$. By Theorem \ref{Iteration2}, for any $q_i\not\mid(K^2+4)$, $q_i$ vanishes for some sufficiently large $N$, and $q_i\not\mid\pi_K(m)$. Hence, for some $N$, $\pi_K^n(m)=\lcm\left[6,4p_1^{j_1},\ldots,4p_r^{j_r}\right]$. If there is at least one such $p_i$, then $\pi_K^n(m)=12\cdot p_1^{j_1}\cdots p_r^{j_r}$ where $p_1^{a_1}\cdots p_r^{a_r}$ is the prime factorization of $K^2+4$; otherwise, $\pi_K^n(m)=6$. Finally, since $K\equiv3\pmod{6}$, there is a 2-periodic point: $\pi_K(2)=3\rightarrow\pi_K(3)=2\circlearrowleft$.
\end{proof}

\begin{theorem}
    If $K\equiv2\ \text{or}\ 0\pmod{4}$ and $m>1$, then there exists an $N$ such that for all $n\vargeq N$, $\pi_K^n(m)=p_1^{j_1}\cdots p_r^{j_r}$ for $j_i=0,1,2,\ldots$, where $p_1^{a_1}\cdots p_r^{a_r}$ is the prime factorization of $K^2+4$.
\end{theorem}

\begin{proof}
    Suppose $K\equiv 2\ \text{or}\ 0 \pmod4$ and suppose $m=p_1^{j_1}\cdots p_r^{j_r} q_1^{\lambda_t}\cdots q_t^{\lambda_t}>1$ for $j_i\vargeq0$ and $\lambda_i\vargeq1$ (and $m>3$ if $K\equiv3\pmod{6}$), where $p_1^{a_1}\cdots p_r^{a_r}$ is the prime factorization of $K^2+4$. By Theorem \ref{lcm}, 
    \begin{align}\label{K=2,0mod4_P1}
    \pi_K(m)=\lcm\left[\pi_K(p_1^{j_1}),\ldots, \pi_K(p_r^{j_r}), \pi_K(q_1^{\lambda_1}),\ldots, \pi_K(q_s^{\lambda_2})\right]
    \end{align}
    By Theorem \ref{Iteration1}, if any $j_i\vargeq1$, then the corresponding $p_i$ is preserved for all $N$, and $p_i\mid\pi_K^N(m)$. By Theorem \ref{Iteration2}, for any $q_i\not\mid(K^2+4)$, $q_i$ vanishes for some sufficiently large $N$, and $q_i\not\mid\pi_K^N(m)$. Hence, there is a positive integer $N$ such that for all $n\vargeq N$, $\pi_K^n(m)=p_1^{j_1}\cdots p_r^{j_r}$ where $p_1^{a_1}\cdots p_r^{a_r}$ is the prime factorization of $K^2+4$. 
\end{proof}

\begin{section}{Proof of Theorem \ref{trajectory}}

Let $q$ be an odd prime factor of $m$ that is coprime to all fixed points, and define $S(m)$ by
\[
S(m) = \sum_{q\mid m} = \nu_q(m) (\log(q)-\log(3)) = \log m - \sum_{q \mid m} \nu_q(m)\log(3).
\] where $\nu_q(m)$ is the $q$-adic valuation of $m$. By Theorem \ref{quad_res}, $\pi_K(q) \varleq 2(q+1) \varleq \frac{8}{3}q$. Since this number is even if $q > 2$, $\pi_K(q)$ must have at least two prime factors. For arbitrarily large $q$, 
\[
S(\pi_K(q)) \varleq \log(q)-2\log(3) = S(q)+\log(8/3)-\log(3) = S(q)+3\log(2)-2\log(3).
\]
Since $\pi_K(q^e) \mid q^{e-1}\pi_K(q)$, $S(\pi_K(q^e)) < S(q^e)+3\log(2)-2\log(3)$. Hence $S=0$ is reached in at most $\frac{\log m}{2\log(3)-3\log(2)}$ steps, where one step represents the calculation of the Pisano period for a certain $q_i^{e_i}$. While one iteration of the Pisano period always accomplishes at least one step, as any $q_i^{e_i}$ term may be chosen and separated out by Theorem \ref{lcm}, it may accomplish more than one step if multiple $q_i$'s are present.

Any number $m$ where $S(m)=0$ must be of the form $2^a3^b\prod_{i=1}^t p_i^{j_i}$, where $p_i \mid K^2+4$ for all $i$ as seen in the previous section. Note that by Theorem \ref{quad_res}, all factors of $p_i^{j_i}$ remain constant upon iteration of the period, as $\pi_K(p_i^{j_i})=4p_i^{j_i}$, and the only factors that vary are powers of $2$ and $3$. If $m'$ is not a multiple of a fixed point, then the Pisano period may be applied recursively to achieve a fixed point. The maximum number of such iterations is 5, when $K \equiv \pm1 \pmod 6$.

Now suppose $m'$ is a multiple of a fixed point but not a fixed point itself. The only primes $p$ whose $p$-adic valuations may change on the iteration of the Pisano period when $S=0$ are 2 and 3; hence let $g(m')=\nu_2(m')+\nu_3(m')$. Suppose $K \equiv \pm1 \pmod 6$. Then, $\nu_2(\pi_K(2^a))=a-1$, $\nu_2(\pi_K(3^b))=3$, and $\nu_2(\pi_K(p_i^{e_i}))=2$; likewise, $\nu_3(\pi_K(2^a))=1$, $\nu_3(\pi_K(3^b))=b-1$, and $\nu_3(\pi_K(p_i^{e_i}))=0$. Since $m$ is a multiple of a fixed point, $\nu_2(m') \vargeq 3$ and $\nu_3(m') \vargeq 1$, and because $m'$ is not a fixed point itself there must be at least one additional factor of either 2 or 3. If $\nu_2(m') > 3$, then $\nu_2(\pi_K(m')) = \nu_2(m') - 1$, and if $\nu_3(m') > 1$, then $\nu_3(\pi_K(m')) = \nu_3(m') - 1$. Hence $g(m')$ decreases by at least 1 until a fixed point is reached, and may not be greater than $\log_2m' = \frac{\log m}{\log(2)} + \frac{\log m}{2\log(3)-3\log(2)} + 5$. The maximum number of iterations of the Pisano period necessary to reach a fixed point is thus $\frac{\log m}{\log(2)} + \frac{\log m}{2\log(3)-3\log(2)} + 1$.

If $K \equiv 3 \pmod 6$, then the trajectory of $m'$ will end at either the 2-periodic point $(2,3)$ or a fixed point that is a multiple of 12. If $p_i \mid m$ for at least one $p_i \not \in \{2,3\}$, then $12 \mid \pi_K^N(m')$ for $N \vargeq 2$ since $\pi_K(p_i) = 4p_i$ and $3 \mid \pi_K(4)$. As above, $\nu_2(\pi_K(2^a))=a-1$, $\nu_2(\pi_K(p_i^{e_i}))=2$, $\nu_3(\pi_K(3^b)) = b - 1$, and $\nu_3(\pi_K(p_i^{e_i}))=0$, but $\nu_2(\pi_K(3^b))=1$. If the trajectory of $m'$ ends at a multiple of 12 and $g(m) > 3$, then $g(m)$ will again decrease by at least 1 until a fixed point is reached, and the maximum number of iterations is $\frac{\log m}{\log(2)} + \frac{\log m}{2\log(3)-3\log(2)} + 2$. Two extra steps may be necessary if the trajectory terminates at the 2-periodic point.

If $K \equiv 2 \pmod 4$, then any number $m'$ with $S=0$ will reach a fixed point after at most one more iteration, since any combination of powers of $2$ and $p_i$'s is a fixed point.

If $K \equiv 0 \pmod 4$, then $\pi_K(2^a) \mid 2^{a-1}$ for all $a$; hence, the powers of $2$ vanish in at most $\frac{\log m}{\log 2}$ steps, giving $\frac{\log m}{\log(2)} + \frac{\log m}{2\log(3)-3\log(2)}$ total iterations.

Combining these yields an overall upper bound of $\frac{\log m}{\log(2)}+\frac{\log m}{2\log(3) - 3\log(2)} + 2$ on the number of iterations of the Pisano period required to reach a fixed point; hence 
\[
\limsup \frac{\mathcal{T}_K(m)}{\log m} \varleq \frac{1}{\log 2}+\frac{1}{2\log(3)-3\log(2)} \approx 9.933.
\]

\begin{corollary}\label{fixedpointbound}
    Let $\mathcal{P}_K(m)$ be the fixed point that terminates the trajectory of $m$; if $K \equiv 3 \pmod 6$ and $m$ terminates at the 2-periodic point, define $\mathcal{P}_K(m)=0$. Then $\limsup \frac{\log \mathcal{P}_K(m)}{\log m} < \infty$.
\end{corollary}
\begin{proof}
    The first part of the proof of Theorem \ref{trajectory} shows that there can be at most $\frac{\log m}{2\log(3)-3\log(2)}$ steps in the reduction to $S=0$, and by Theorem \ref{quad_res}, 
    \[
    \frac{\pi_K(q_i^{e_i})}{q_i^{e_i}} \varleq \frac{2q+1}{q} \varleq \frac{8}{3},
    \]
    since the minimum possible $q_i$ is 3. If $m'=\pi_K^n(m)$ is such that $S(m')=0$, and $m'$ is a multiple of a fixed point, each successive application of the Pisano period reduces $m'$. If $m'$ is not a multiple of a fixed point, it is multiplied by no more than 24 until a fixed point is reached. Therefore, 
    \[
    \log \mathcal{P}_K(m) \varleq (\log(8)-\log(3))\left(\frac{\log m}{2\log(3) - 3\log(2)}\right)+\log(24)
    \]
    and
    \[\limsup \frac{\log \mathcal{P}_K(m)}{\log m} \varleq \frac{\log(8)-\log(3)}{2\log(3)-3\log(2)} \approx 8.327.
    \]
\end{proof}

\end{section}

\section{Final Thoughts}
A general binary recurrence sequence $\mathcal{U}_{n}$ is determined by the four parameters $(a,b,c,d)$ where 
\[
\mathcal{U}_0=c, \qquad \mathcal{U}_1=d, \qquad \mathcal{U}_n=a\times\mathcal{U}_{n-1}+b\times\mathcal{U}_{n-2}.
\]
For example, the Fibonacci sequence is determined by $(1,1,0,1)$, the Lucas sequence by $(1,1,2,1)$, the Pell sequence by $(2,1,0,1)$, the Jacobsthal sequence by $(1,2,0,1)$, etc. When the initial values are $0$ and $1$, the recurrence is called an $(a,b)$-Fibonacci sequence. In all of these cases, the binary recurrence sequence is periodic modulo a positive integer $m>1$ \cite{[Everest]}. Naturally, the Pisano period has been considered for many binary recurrence sequences. Is it possible to determine the fixed points for general binary recurrence sequences?
\begin{table}[H]
\begin{tabular}{l|l}
Sequence&Pisano Periods for $m=2,3,\ldots,24$\\
\hline
Fibonacci&3,8,6,20,24,16,12,24,60,10,24,28,48,40,24,36,24,18,60,16,30,48,\textbf{24},\ldots\\
Lucas&3,8,6,4,24,16,12,24,12,10,24,28,48,8,24,36,24,18,12,16,30,48,\textbf{24},\ldots\\
Pell&\textbf{2},8,\textbf{4},12,8,6,\textbf{8},24,12,24,8,28,6,24,16,\textbf{16},24,40,12,24,24,22,8\ldots\\
Jacobsthal&6,2,4,6,\textbf{6},2,18,4,10,6,12,6,12,2,8,18,\textbf{18},4,6,10,22,6,\ldots\\
\end{tabular}
\end{table}
It is well-established that the Pisano period of the Lucas sequence differs from the Pisano period of the Fibonacci sequence only when $m$ is a multiple of $5$ \cite{[A106291]}. And note that the Pell sequence is simply the $K$-Fibonacci sequence when $K=2$. 
\begin{corollary}
    In the Lucas sequence, $m>1$ is a fixed point if and only if $m=24$.
\end{corollary}

\begin{corollary}
    In the Pell sequence, $m>1$ is a fixed point if and only if $m=2^k$ for $k=1, 2, \ldots$
\end{corollary}
\noindent The Jacobsthal sequence is particularly interesting as a different type of sequence where $b\neq1$. Note that the results of Renault used in the proofs of Theorem \ref{K_Fix_Point} used various aspects of $a^2+4b$ where $a=K$ and $b=1$. In the Jacobsthal sequence, note that $a^2+4b=9=3^2$.
\begin{conjecture}
    In the Jacobsthal sequence, $m>1$ is a fixed point if and only if $m=2\cdot3^k$ for $k=1, 2, \ldots$
\end{conjecture}

Renault's Theorem \ref{quad_res} is originally stated for $(a,b)$-Fibonacci sequences, which is less controllable when $b\neq1$. Let $\pi_{(a,b)}(m)$ denote the $(a,b)$-Pisano period of the $(a,b)$-Fibonacci sequence modulo $m$.  
\begin{theorem}[Renault \cite{Renault}]
Let $p$ be an odd prime such that $p\not\mid b$. Then,
\begin{itemize}
    \item if $a^2+4b$ is a nonzero quadratic residue modulo $p$, then $\pi_{(a,b)}(p)\mid(p-1)$\\
    \item if $a^2+4b$ is a quadratic nonresidue, then $\pi_{(a,b)}(p) \mid (p+1)\text{ord}_p(-b)$ except when $b=-1$, in which case $\pi_{(a,b)}(p)\not\mid(p+1)$.\\
    \item if $p\mid(a^2+4b)$, then $\pi_{(a,b)}(p)=p\cdot\text{ord}_p(2^{-1}\cdot a)$.
\end{itemize}
\end{theorem}
An outlying case occurs when $b=-1$. There are five degenerate cases: $a=\pm2$ and $a=\pm1$ and $a=0$. These occur when the ratio of the roots of the characteristic polynomial is a root of unity or if $a^2+4b=0$.

If $a=1$, then $a^2-4b=-3$, and the sequence is a repeating six-term cycle: $0, 1, 1, 0, -1, -1, \ldots$. Note that if $m=2$, then $\pi_{(1,-1)}(2)=3$ but for all other $m>2$, $\pi_{(1,-1)}(m)=6$. Hence, the only fixed point is $m=6$. Similarly, if $a=-1$, then $a^2-4b=-3$, and the sequence is a repeating three-term cycle: $0, 1, -1,\ldots$. Note that modulo any $m>2$, the $(-1,1,0,1)$-sequence always has a period of 3. Hence, the only fixed point is $m=3$.

If instead, $a=2$, then $a^2+4b=0$ and the sequence is exactly: $0, 1, 2, 3, 4, 5, \ldots$, hence every positive integer $m>1$ is a fixed point, since modulo $m$ the sequence repeats with a length of $m$. Somewhat similarly, if $a=-2$, then $a^2+4b=0$ and $\pi_{(-2,-1)}(m)=m$ whenever $m$ is even and $\pi_{(-2,-1)}(m)=2m$ whenever $m$ is odd. Then, fixed points occur whenever $m$ is even.

And if $a=0$, then $a^2-4b=-4$, and the sequence is a repeating four-term cycle: $0, 1, 0, -1, \ldots$. Note that if $m=2$, then $\pi_{(0,-1)}(2)=2$ and for all other $m>2$, $\pi_{(0,-1)}(m)=4$. Hence, the only fixed points are $m=2$ and $m=4$.

For other values of $a$, $a^2+4b>0$ and something more familiar occurs. 
\begin{conjecture}
    Let $p_1^{e_1}p_2^{e_2}\cdots p_t^{e_t}$ be the prime factorization of $a^2+4b$ for $a>0$ and $b=-1$. Suppose $a>2$ then, for every modulus $m>1$ and for $j_i=0,1,2,\ldots$, fixed points $\pi_{(a,-1)}(m)=m$ occur if and only if
    \begin{enumerate}[label=(\roman*)]
        \item $a\equiv1\pmod6$ and $m=p_i^{j_i}$ where $p_i \mid m$ or $m=6\cdot p_1^{j_1}\cdots p_t^{j_t}$\\
        \item $a\equiv2\pmod4$ and $m=p_i^{j_i}$ where $p_i \mid m$ or $m=2^{j_1+1}\cdot p_2^{j_2}\cdots p_t^{j_t}$ where $p_1=2$\\
        \item $a\equiv3\pmod6$ and $m=p_i^{j_i}$ where $p_i \mid m$ and $m=12\cdot p_1^{j_1}\cdots p_t^{j_t}$ for $a>3$, \\
        and $m=12\cdot p_1^{j_1}\cdots p_t^{j_t}$ when $a=3$\\
        \item $a\equiv0\pmod4$ and $m=p_i^{j_i}$ where $p_i \mid m$ or $m=2\cdot p_2^{j_2}\cdots p_t^{j_t}$ or $m=4\cdot p_2^{j_2}\cdots p_t^{j_t}$ where $p_1=2$\\
        \item $a\equiv-1\pmod6$ and $m=p_i^{j_i}$ or $m=6\cdot p_1^{j_1+1}\cdots p_t^{j_t}$.
    \end{enumerate}
    Suppose instead that $a<-1$, then for every modulus $m>1$ and for $j_i=0,1,2,\ldots$, fixed points $\pi_{(a,-1)}(m)=m$ occur if and only if
    \begin{enumerate}[label=(\roman*)]
        \item $a\equiv1\pmod6$ and $m=p_i^{j_i}$ or $m=2^{j_1}\cdot3^{j_2+1}\cdots p_t^{j_t+1}$ where $p_1=3$\\
        \item $a\equiv2\pmod4$ and $m=p_i^{j_i}$ where $p_i \mid m$ or $m=2^{j_1+1}\cdot p_2^{j_2}\cdots p_t^{j_t}$\\
        \item $a\equiv3\pmod6$ and 
        $m=p_i^{j_i}$ where $p_i \mid m$ or $m=12\cdot p_1^{j_1}\cdots p_t^{j_t}$\\
        \item $a\equiv0\pmod4$ and $m=p_i^{j_i}$ where $p_i \mid m$ or $m=2\cdot p_2^{j_2}\cdots p_t^{j_t}$ or $m=4\cdot p_2^{j_2}\cdots p_t^{j_t}$ where $p_1=2$\\
        \item $a\equiv-1\pmod6$ and $m=p_i^{j_i}$ where $p_i \mid m$ or $m=2\cdot3\cdot p_2^{j_2}\cdots p_t^{j_t}$ where $p_1=3$.
    \end{enumerate}
\end{conjecture}

It appears that in the previous conjecture, only one prime has powers that are fixed points for any given $a$; let this be the \textit{critical prime} for the $(a,-1)$-Fibonacci sequence. However, this prime is not always the smallest prime that divides the $a^2+4b$, nor always the largest. What is the behavior of the critical prime in general? A singular case occurs when $a=3$, there is no critical prime; $a^2+4b=5$, but powers of $5$ are not fixed points.

\section{Acknowledgements}
The authors would like to extend sincerest appreciation to Dr. Florian Luca for the suggested addition of Theorem \ref{trajectory} and Corollary \ref{fixedpointbound}, and to Michael De Vlieger for their invaluable assistance in writing the code used in the experimental stages of this research.

\bibliographystyle{plain}
\bibliography{fixedpoints}{}

\end{document}